\documentclass[12pt]{article}
\usepackage[a4paper, margin=0.98in]{geometry}
\usepackage{amsmath,amssymb,amsthm,graphicx,fixmath,latexsym,color}
\usepackage{hyperref}
\usepackage{amstext}
\usepackage{bbm}

\newtheorem{thm}{Theorem}[section]
\newtheorem{lem}[thm]{Lemma}

\theoremstyle{definition}

\newtheorem{claim}[thm]{Claim}

\renewcommand{\Re}{\mathbb R}
\renewcommand{\epsilon}{\varepsilon}

\newcommand{\C}{\mathcal C}

\renewcommand{\phi}{\varphi}

\newcommand{\noshow}[1]{}

\newcommand{\cardin}[1]{\lvert {#1} \rvert}
\newcommand{\E}{\mathcal{E}}

\newcommand{\cf}{{{\chi_{\text{\textup{cf}}}}}}

\begin{document}
\title{Hypergraphical Clustering Games of Mis-Coordination%
\footnote{An earlier version of this paper was circulated under the title ``The Price of Anarchy in Hypergraph Coloring Games".} }
\author{Rann Smorodinsky\thanks{Faculty of Industrial Engineering and Management, Technion Israel. rann@ie.technion.ac.il. Research was partially supported by the United States-Israel Binational Science Foundation and National Science Foundation grant 2016734, the German-Israel Foundation grant I-1419-118.4/2017, the Ministry of Science and Technology grant 19400214, Technion VPR grants and the Bernard M. Gordon Center for Systems Engineering at the Technion.} \and Shakhar
Smorodinsky\thanks{Department of Mathematics, Ben-Gurion University of the
NEGEV, Be'er-Sheva Israel. shakhar@math.bgu.ac.il. Research was partially supported by Grant 635/16
from the Israel Science Foundation.}}

\date{}
\maketitle

\begin{abstract}
We introduce and motivate the study of hypergraphical clustering games of mis-coordination. For two specific variants we prove the existence of a pure Nash equilibrium and provide bounds on the price of anarchy as a function of the cardinality of the action set and the size of the hyperedges.
\end{abstract}

\section{Introduction}

In many strategic scenarios players are involved in a variety of games simultaneously and players are not necessarily involved in all games. When these games are interrelated in the sense that players choose their action simultaneously for all games  and a player's overall utility is the sum over the utilities in each game then the collection of games forms a succinct representation of a larger game. Papadimitriou and Roughgarden \cite{Papadimitriou:2008:CCE:1379759.1379762} coin the term `hypergraphical game' for such games. Given a hypergraph, each vertex represents a player while each hyperedge represents the set of participants of a constituent game.

If, in addition, each constituent game (over a hyperedge) is symmetric and respects permutation over the action set (renaming of actions does not matter) we refer to such games as {\em hypergraphical clustering games}. For hypergraphical clustering games we will often refer to actions as colors and a cluster is just a set of players that share the same color.
Restricting attention furthermore to utility functions that hinge on whether a player's action (or color) is different from those chosen by his peers we get {\em hypergraphical clustering games of mis-coordination}.%
\footnote{In a similar vein, when a player's utility function hinges on whether the player's color is the same as those chosen by his peers we get {\em hypergraphical clustering games of coordination}.}

Hypergraph clustering games of mis-coordination come up in a variety of settings:
\begin{itemize}
\item
Consider a frequency allocation setting, where individual radio stations (the vertices) must choose their broadcasting frequency (color). Consider an arbitrary listener at an arbitrary geographical point. This listener is within communication range of a subset of the stations (a hyperedge). The broadcast quality of a station to the listener depends on how different its frequency is from those chosen by others who are within the listener's range.
\item
An individual  belongs to various social circles (family, work, childhood friends, etc.) and must choose some property from some given finite collection, say his profession or address. He derives utility from each social circle depending on whether or not his choice is the same as others in the circle. It is quite natural to assume that an individual would prefer to be co-located with peers from the various circles (coordination) but may want to stand out in his profession (mis-coordination).
 \item
A firm is involved in various markets (horizontally or vertically). The firm must choose a branding strategy (equivalently, an accounting policy, a critical supplier or an IP strategy).  The utility from each of these depends on how close or far its choice is from the competition in each of the relevant markets. This is a hypergraph clustering game where each vertex represents a firm, a hyperedge represents a market and the set of colors is the set of available branding strategies.%
\footnote{Anecdotally, a critical component of a branding strategy in consumer facing markets are actual colors. For example, the cellular service providers in Israel are each associated with its unique color.}
\item
In an aerial transportation problem, each plane traverses space in a certain route composed of segments. Each plane must choose a cruising altitude with a clear preference not to share its altitude with other planes that simultaneously share the segment. In this setting, vertices represent planes, hyperedges represent segments and colors represent altitudes.
\end{itemize}

When we further restrict attention to models where the underlying constituent games (the hyperedges) are two-player games we refer to these as `clustering games'. This class has been studied in the recent decade and we further discuss this in Section \ref{lit_survey}.
Defining mis-coordination in the simple graph setting is straightforward. A vertex (player) gains utility only at edges where the corresponding neighbor chose a different color. Moving from edges to hyperedges allows for some flexibility in the notion of mis-coordination. To make the distinction from the graph setting let us refer to vertices  who seek to mis-coordinate their color as homogeneity-averse agents.

We consider two variants of homogeneity-aversion:%
\begin{itemize}
\item
A weak notion of homogeneity-aversion is when each vertex gains, at a given hyperedge, when there is at least one other vertex in the hyperedge with a different color.
\item
A strong notion of homogeneity-aversion is when a vertex gains, at a given hyperedge, when its color is different from the colors of all other vertices in that hyperedge.
\end{itemize}

Note that both notions agree over simple graphs.

In this paper we study the price of anarchy (PoA) of hypergraph clustering games of mis-coordination. The PoA, introduced in \cite{DBLP:journals/csr/KoutsoupiasP09}, has become the de-facto standard for measuring the inefficiency resulting from delegating the decision on color choice to the players (the vertices) as opposed to centrally assigning them.
Since its introduction the PoA has been extensively studied for many families of games. To the best of our knowledge this is the first paper to study the PoA in hypergraphical games.

\subsection{Related literature}
\label{lit_survey}

The literature on hypergraphical games and in particular hypergraph clustering games (going beyond simple graphs) is very limited. As already mentioned, Papadimitriou and Roughgarden \cite{Papadimitriou:2008:CCE:1379759.1379762} introduce the notion of `hypegraphical games' and provide computational results regarding correlated equilibria in such games. Chung and Tsiatas \cite{ChungTsiatas2012} study the limit outcome of various types of best-reply dynamics where at each stage some hyperedge is chosen and its members change their color probabilistically. Simon and Wojtczak \cite{ijcai2017-57} prove the existence of a strong equilibrium in a class of hypergraph clustering games they refer to as synchronization games. They go on and study the complexity of computing Nash and strong equilibria. Synchronization games are quite similar to one of our versions of mis-coordination games with some exceptions. For example, in such games each vertex may have a different color palette at his disposal.

Much more work has been done in the context of (simple) graphs (see \cite{CKPS10,CGJ08,KSS15,MBQ16,PS12} for a sample of related papers). The class of games over graphs (or networks) was first suggested as a model of wide interest by Kearns et al \cite{KearnsEtAl2001}.
Hoefer \cite{Hoefer2007} focuses on games of coordination, where the utility function, at an edge, takes on a value of 1 when the corresponding colors are the same and 0 otherwise and recently \cite{Apt2017} studies the PoA and the strong PoA for this class of games.%
\footnote{The {\em Strong PoA} is the ratio between the social welfare obtained in the optimal outcome and that obtained in the worst case {\em strong} equilibrium.}
In Anti-coordination Games (often refered to as Max $k$-Cut games) agents get a utility for each edge where colors differ. The PoA and the Strong PoA in such games has been studied in \cite{KunPR13, Gourves2009, Gourves2010}.
Feldman and Friedler  \cite{Feldman_Friedler_2015} devise a unified framework for games of coordination and anti-coordination, a class of games they refer to as `clustering games', and provide stronger as well as new results for the PoA and the strong PoA of such games.%
\footnote{The notion of a `hypergraphical clustering game', introduced above, naturally extends the notion of a `symmetric clustering game' introduced in \cite{Feldman_Friedler_2015}.}

Optimal (centralized) coloring schemes have been studied in the combinatorics and computer science communities. One prominent example is the frequency allocation problem. For a detailed discussion of the problem and related results we refer the reader to a survey written by the second author \cite{CF-survey}.

\subsection{Organization of the paper}
The formal model and the two utility functions we discuss are introduced in Section~\ref{sec:model}. The two main theorems are stated in Section \ref{subsec-utilities} while Sections  \ref{sec:non-monochromatic} and \ref{sec:conflict-free} consist of the corresponding proofs. Section~\ref{sec:discussion} discusses the results as well as natural directions for future research.

\section{Model}
\label{sec:model}

A {\em hypergraph} is a pair
$(V,\E$) where $V$ is a set and $\E$ is a collection of non-empty subsets
of $V$. The elements of $V$ are called {\em vertices} (alternatively, players or agents) and the elements of $\E$ are called {\em hyperedges} (alternatively, coalitions). When all hyperedges
in $\E$ contain exactly two elements of $V$ then the pair
$(V,\E)$ is a {\em simple graph}.
$H=(V,\E)$ is called {\em $r$-uniform} if $|e|= r$ for all $e \in \E$. So, in this terminology, a $2$-uniform hypergraph is a simple graph. We say that $H=(V,\E)$ is {\em $r$-minimal} if $|e|\ge r$ for all $e \in \E$. We denote by $\mathbb{H}_r$ the set of all $r$-uniform hypergraphs
and by $\mathbb{H}_{\geq r}$ the set of all $r$-minimal hypergraphs.

A {\em $k$-coloring} (alternatively, a strategy tuple) of $H$ is a function $c: V \rightarrow [k]$. We denote by $\mathcal{C}(k)$ the set of all $k$-colorings. For a given $k$-coloring $c$ and a hyperedge $e \in \E$, let $c(e) \subset [k]$  be the image of $e$, i.e., the set of colors associated with the vertices in $e$.

Each vertex (player) is endowed with a utility function,
$u_v: {\mathcal{C}(k)} \rightarrow \mathbb{R}$.
The social welfare of a $k$-coloring $c$, denoted $SW(c) = \sum_{v \in V}u_v(c)$, is the sum of the agents' utilities.

In this paper we focus on games of mis-coordination. In the case of simple graphs the notion of mis-coordination describes settings where agents would like to mis-coordinate with their neighbors. In particular each agent receives a utility of $1$ for each edge it participates in and where the color chosen by the neighbor is different than its own. Extending this concept to hypergraphs is ambiguous. Let us denote
by $\E(v) = \{e\in \E: v \in e\}$ the set of hyperedges containing $v$.
We propose two natural extensions:

\begin{enumerate}
\item
A vertex is called {\em non-monochromatic seeking} (an {\em NM-vertex} in short) if its utility function is given by
$$u_v(c)=|\{e\in \E(v): |c(e)|>1  \}|.$$
In words, a vertex enjoys each  hyperedge in which its color differs from at least one other vertex. Put differently, the utility function of $v$ is the number of non-monochromatic hyperedges
containing $v$.
\item
A vertex is called {\em conflict-free seeking} (a {\em CF-vertex} in short) if its utility function is given by
$$
u_v(c)=\cardin{ \{e \in \E(v) \mid |c(e)|=|c(e \setminus \{v\})| + 1 \} }.$$
In words, $v$ enjoys each hyperedge for which its color differs from the colors of all other vertices in that hyperedge.
\end{enumerate}
Note that for simple graphs (i.e., $r=2$) the two utility functions coincide.\\

A coloring $c$ is called a {\em (pure) Nash equilibrium} if no player can increase her utility by a unilateral change.
Formally, $u_v(c) \ge u_v(c_{-v},i)$ for all $v \in V$ and $i \in [k]$ (where $(c_{-v},i)$ denotes the coloring $c$ with agent $v$ substituting her color to $i$).

For a hypergraph $H$ and an integer $k \geq 2$, put $O(H,k)=\max \{SW(c) \mid c \in \mathcal{C}(k)\}$.
Put $NE(H,k) = \min \{SW(c) \mid c \in \mathcal{C}(k)\text{ is a Nash Equilibrium of } H  \}$. When the number of colors $k$ is known and is clear from the context, we sometimes abuse the notation and write $O(H)$ and $NE(H)$ instead.

For a given integer $k \geq 2$  and a given family of hypergraphs ${\cal H}$ we define the {\em Price of Anarchy} as
$PoA=PoA({\cal H},k) = \sup_{(H\in {\cal H})}\frac{O(H,k)}{NE(H,k)}$.

\section{Results}
\label{subsec-utilities}

Each of the two utility functions induces a game. We will show that each such game is in fact a potential game and hence admits a pure Nash equilibrium. For brevity we remind the reader of the notion of a potential game:

\noindent {\bf Potential function:}
Let $H=(V,\E)$ be a hypergraph and let $u_v(c)$ be a utility function.
A {\em potential function} for $H$ is a function $\psi:\C(k) \rightarrow \Re$ such that for any two colorings $c,c' \in \C(k)$ if $c$ and $c'$ differ only on one vertex $v \in V$. Then $$
\psi(c)- \psi(c')= u_v(c)-u_v(c').
$$

The following lemma is well known (see, e.g., \cite{MS96}):
\begin{lem}
\label{potential-lem}
Let $H$ be a hypergraph and $u$ a utility function. If $H$ admits a potential function $\psi$ with respect to $u$ then there exists a pure Nash equilibrium for the corresponding coloring game.
\end{lem}

We provide upper and lower bounds on the price of anarchy, for each of the utility functions,  as a function of two parameters: The number of available colors $k$ and the hyperedge size, $r$.

\begin{thm}
 \label{poa-NM}
 Whenever players are non-monochromatic seeking:
 $$
 PoA{(\mathbb{H}_{\geq r},k)} = (1+\frac{1}{(k-1)r}) \ \
  \forall r\geq 3, k \geq 2.
 $$
 \end{thm}

\begin{thm}
\label{poa-cf-new}
 Whenever players are conflict-free seeking:
 \mbox{}
\begin{enumerate}
\item
 $(\frac{k}{k-1})^{r-1} \leq PoA(\mathbb{H}_{r},k) \leq  \frac{2k+r-2}{2k-r} $  whenever $k \geq r$.%
 \footnote{Note that for a simple graph, namely when $r=2$, the utility function identifies with that studied in \cite{KunPR13}. In addition, the lower and upper bounds equal each other as well as to $\frac{k}{k-1}$, the bound obtained in \cite{KunPR13}.}
\item
 $ \frac{k-1}{r} \left (\frac{k}{k-1} \right )^{r-1} \leq PoA(\mathbb{H}_{r},k) \leq \frac{k-1}{r}\frac{2k+r-2}{2k-r} $
  whenever $\frac{r}{2} < k < r$.
\item
  $PoA(\mathbb{H}_{r},k) =\infty$ whenever  $k \leq\frac{r}{2}$.
\end{enumerate}
\end{thm}

The following two sections contain the proofs for both theorems.

\section{Non-monochromatic seeking agents}
\label{sec:non-monochromatic}

Throughout this section we assume vertices are  NM-vertices, that is their utility is given by $u_v(c)=|\{e\in \E(v): |c(e)|>1  \}|.$

\begin{claim}
\label{thm_mono is potential}
For an integer $k$, a hypergraph $H=(V,\E)$ and a coloring $c \in \C(k)$, let $\psi(c)$ be the number of non-monochromatic hyperedges in $\E$. Then $\psi$ is a potential function for the corresponding hypergraphical clustering game.
\end{claim}
The proof of Claim~\ref{thm_mono is potential} is straightforward and hence omitted.
Note that Lemma~\ref{potential-lem} combined with Claim~\ref{thm_mono is potential} implies that the game played among NM-vertices admits a pure Nash equilibrium and, furthermore, that the best Nash equilibrium attains the social optimum. However, other Nash equilibria may entail low social welfare.

Before we turn to the proof of Theorem \ref{poa-NM} we require some notations and a lemma.

For an $r$-minimal hypergraph $H=(V,\E)$, a $k$-coloring $c$ of $H$ and a vertex $v \in V$ we define the following four parameters:
\begin{enumerate}
\item
For any $i\not = c(v)$ let $d_1^i(v) = |\{e \in \E(v) \mid \cardin{c(e)}=2, c(e\setminus \{v\})=\{i\} \  \}|$ be the number of hyperedges $e \in \E$ incident to $v$ for which the color of all other vertices in $e$ is $i$ which is different from $c(v)$. Let $d_1(v) = \sum_{i\not = c(v)}d_1^i(v)$.
\item
$d_2(v) = |\{e \in \E(v) \mid \cardin{c(e)}=1 \}| $ is the number of monochromatic hyperedges containing $v$.
\item
$d_3(v) = |\{e \in \E(v) \mid \cardin{c(e)}=2, \exists v'\not = v \mid c(e\setminus \{v'\})|=1 \}|$ is the number of hyperedges incident to $v$ with exactly two colors, of which one vertex distinct from $v$ has a unique color.
\item
$d_4(v) = \cardin{\{e \in \E(v)\}} - \left( d_1(v)+d_2(v)+d_3(v) \right)$  is the number of hyperedges containing $v$ that do not fall into any of the first three categories.
\end{enumerate}

We denote by  $D_i = D_i(c) = \sum_{v \in V} d_i(v)$ the corresponding sums.

Note that for $r=2$, the set of hyperedges in $\E(v)$ counted in $d_1(v)$ is identical to the set counted in $d_3(v)$ whereas for $r \geq 3$ those two sets are disjoint.

We need the following lemma:
\begin{lem}\label{auxiliary}
Let $r \geq 3$ and let $H$ be an arbitrary hypergraph in $\mathbb{H}_{\geq r}$. For any coloring $c$ we have:
\begin{enumerate}
\item
$D_1+D_2+D_3+D_4 = \sum_{v\in V}|\E(v)|$.
\item
$D_3 \ge (r-1)D_1$.
\item
$D_1 \geq (k-1)D_2$ whenever $c$ is a Nash equilibrium.
\end{enumerate}
\end{lem}

\begin{proof}
(1) is straightforward. (2) Note that for every vertex $v$ $d_1(v)+d_3(v)$ counts the total number of hyperedges in $e \in \E(v)$ with $\cardin{c(e)}=2$ so that there is vertex $v \in e$ whose color is distinct from all other vertices in $e$. Note also that any hyperedge $e \in \E$ with this property is counted exactly once in some $d_1(v)$ and at least $r-1$ times in $d_3(u)$ for the other vertices $u \in e$. So $D_3 \geq (r-1)D_1$.   As for (3) let $c$ form a Nash equilibrium. Note that the utility of $v$ is $u_v(c) = d_1(v)+d_3(v)+d_4(v)$. Following a deviation of $v$ from $c(v)$ to some other color $i \not = c(v)$ will increase the utility by $d_2(v)$ (the corresponding hyperedges that are monochromatic will cease to be so) but will simultaneously decrease the utility by $d_1^i(v)$. As $c$ is a Nash equilibrium the net increase cannot be positive and so $d_1^i(v) \geq d_2(v)$. Summing over $i \ne c(v)$ gives $d_1(v) \geq (k-1)d_2(v)$ for any $v$. The asserted inequality follows by summing over all $v$.
\end{proof}

We now turn to the proof of Theorem~\ref{poa-NM}:

\begin{proof}[Proof of Theorem~\ref{poa-NM}]

The proof has two parts. First we show that $PoA \le 1+\frac{1}{(k-1)r}$ and then we show that $PoA \ge 1+\frac{1}{(k-1)r}$ by providing an explicit construction of an $r$-uniform hypergraph $H$ and a Nash equilibrium coloring $c$ for which $\frac{O(H)}{SW(c)} = 1+\frac{1}{(k-1)r}$.

{\bf \noindent Upper Bound:}

Let $H$ be an arbitrary hypergraph in $\mathbb{H}_{\geq r}$ and let $c$ be a Nash equilibrium coloring that minimizes the social welfare $SW(c)$ over all Nash equilibria so $NE(H)=SW(c)$. Let $D_i =D_i(c)$. $O(H)$ clearly satisfies the inequality $O(H) \leq \sum_{v\in V}|\E(v)|$ which implies that $O(H) \leq D_1+D_2+D_3+D_4$ by part (i) of Lemma~\ref{auxiliary}. By part (iii) $D_1+D_2+D_3+D_4 \leq D_1+\frac{1}{k-1}D_1+ D_3+D_4$. On the other hand, as noted in the proof of Lemma~\ref{auxiliary},
 $NE(H) = SW(c) = D_1+D_3+D_4$.
Hence, $\frac{O(H)}{NE(H)} \leq \frac{(1+\frac{1}{k-1})D_1+D_3+D_4}{D_1+D_3+D_4}$. As the enumerator is larger than the denominator this expression is monotonically decreasing in $D_3$ and $D_4$. As $D_4\ge 0$ and $D_3 \ge (r-1)D_1$ (part (ii) of Lemma~\ref{auxiliary}) we conclude that $\frac{O(H)}{NE(H)} \leq  \frac{(1+\frac{1}{k-1})D_1+(r-1)D_1}{D_1+(r-1)D_1} = 1+\frac{1}{(k-1)r}$ as asserted.

{\bf \noindent Lower Bound:}
The following construction proves that $PoA \geq (1+\frac{1}{(k-1)r})$:


Let $A_1,\ldots,A_k$ be $k$ pairwise disjoint sets each containing exactly $r-1$ elements.
Put $S= \bigcup_{i=1}^k A_i$. We first construct an auxiliary $r$-uniform hypergraph $G= (S,\E')$ as follows:
$\E'$ consists of all $r$-tuples of the form $\{A_i \cup \{x\} \mid x \in S\setminus A_i, i\in [k]\}$.
So there are a total of $k(r-1)$ vertices and $k(k-1)(r-1)$ hyperedges and each vertex belongs to exactly $r(k-1)$ hyperedges.
We now construct the $r$-uniform hypergraph $H=(V,\E)$ by taking $r$ disjoint copies of $G$ denoted $G_i=(S_i,\E'_i)$ and adding for every original vertex of $v \in S$ a hyperedge consisting of all $r$ copies of $v$. Formally,
$V = \bigcup_{j=1}^r S_j$  so $\cardin{V} = r(r-1)k$.  Put $\E_1 = \bigcup_{j=1}^r\E'_j$. For a vertex $v \in S$, let
$\{v_1,\ldots,v_r\}$ be the set of its copies in $V$. Put $\E_2 = \{ \{v_1,\ldots,v_r\} \mid v \in S\}$. Finally put $\E= \E_1 \cup \E_2$. See Figure~\ref{fig:LB-mono} for an illustration.

\begin{figure}[htb]
    \begin{center}
        \includegraphics[width=10cm,height=5cm]{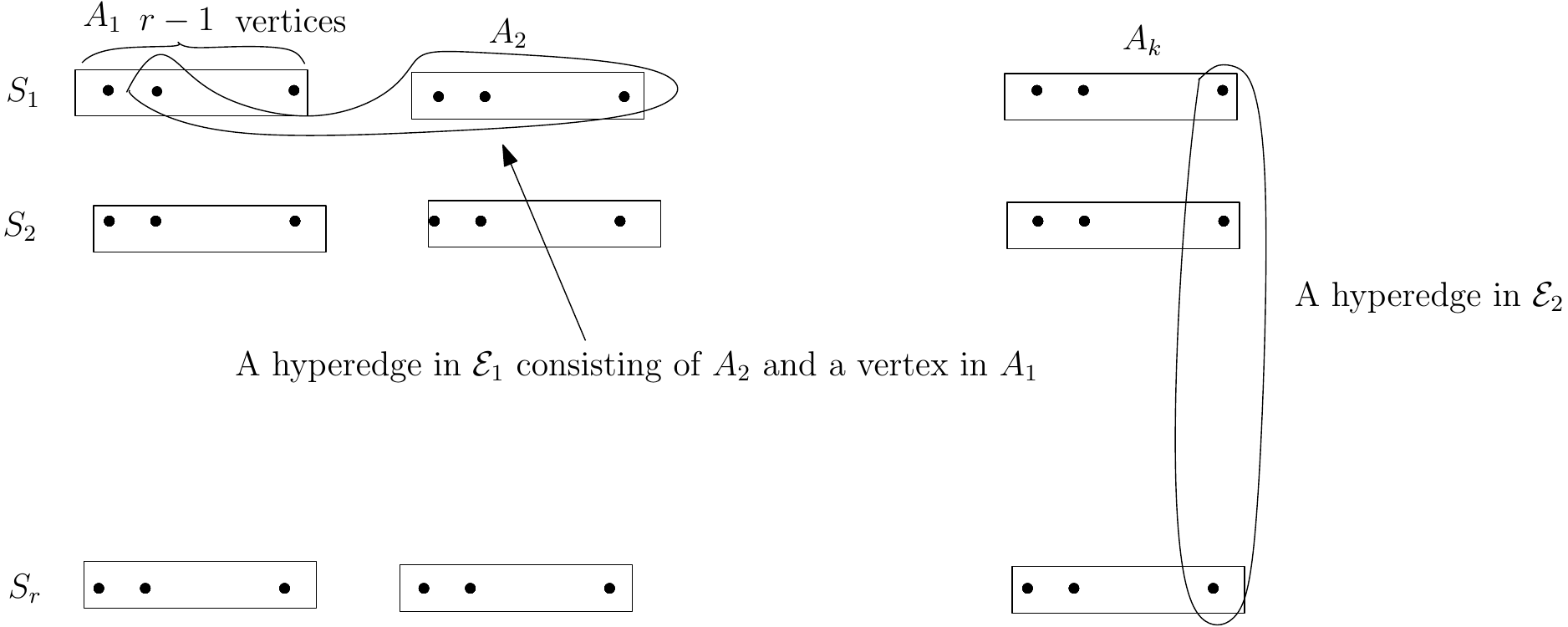}
        \caption{Illustration of the lower bound construction in Theorem~\ref{poa-NM}.}

        \label{fig:LB-mono}
    \end{center}
\end{figure}

Note that each vertex in $V$ belongs to exactly one hyperedge of $\E_2$ and $r(k-1)$ hyperedges of $\E_1$.

Consider the following $k$-coloring $c$.  A vertex $v \in V$ is colored $i$ if and only if it is a copy of vertex in $G$ that belongs to $A_i$. For this coloring the hyperedges in $\E_1$ are all nonmonochromatic. However, those in $\E_2$ are all monochromatic.  Therefore $u_v(c) =  r(k-1)$ and $SW(c) = |V|r(k-1)$. Note that any unilateral change in the color of $v$ adds a nonmonochromatic hypergraph to $\E_2$ but reduces the number of  nonmonochromatic hyperedges in $\E_1$ by 1 as well and hence is not profitable. Therefore $c$ is a Nash equilibrium coloring.

On the other hand we can properly color this hypergraph by taking the coloring $c$ and make a cyclic shift in the colors for $S_1$ where all vertices of $A_i$ are colored with $ i+ 1 \mod k$. Denote the resulting coloring ${\bar c}$.
As this is a proper coloring,  $u_{\bar c}(v) = r(k-1)+1$  for each $v \in V$. This, in turn, implies that
$SW({\bar c}) = |V|(r(k-1)+1)$.

Note that for this hypergraph
$$\frac{O(H)}{NE(H)} =  \frac{\max \{SW({\bar c}) \mid c \in \mathcal{C}(k)\}}{\min \{SW(c') \mid c'\in \mathcal{C}(k) \text{ is a Nash equilibrium } \}}  \geq \frac{r(k-1)+1}{r(k-1)}$$ which proves that
$ PoA(H,k) \ge 1+\frac{1}{(k-1)r}$.

\end{proof}

Theorem \ref{poa-NM} complements results obtained in \cite{KunPR13} who provide a bound of $\frac{k}{k-1}$ for graphs (namely for the case where $r=2$). Note that our bound does not coincide with theirs. The reason is that (as mentioned already) the hyperedges counted by the parameters $d_1(v)$ and $d_3(v)$ coincide over graphs (where all hyperedges are of size $2$) whereas they are pairwise disjoint whenever $r\ge 3$. Once we notice this our technique reaffirms the bound obtained in \cite{KunPR13}.

\section{Conflict-Free seeking agents}
\label{sec:conflict-free}

Throughout this section we assume vertices are CF-vertices, that is, their utility is given by
$u_v(c)=\cardin{ \{e \in \E(v) \mid |c(e)|=|c(e \setminus \{v\})| + 1 \} }$.

For any coloring $c$ and a hyperedge $e\in \E$, put $\phi(c)= \sum_{e\in \E} |c(e)|$.

\begin{claim}\label{claim_phi}
$\phi(c)$ is a potential function for the hypergraphical clustering game played by CF-vertices.
\end{claim}

\begin{proof}
Consider two colorings $c$ and $c'$ that differ only on the vertex $v$. Assume $c(v)=i$ and $c'(v)=j \not =i$. $u_v(c) - u_v(c') = \sum_{e \in \E(v)} \mathbbm{1}(i \not \in c(e\setminus \{v\})) -  \mathbbm{1}(j \not \in c(e\setminus \{v\}))$, where $\mathbbm{1}$ denotes the indicator function.
Note that $|c(e)| = \cardin{c(e\setminus\{v\})} + \mathbbm{1}(i \not \in c(e\setminus \{v\}))$ and similarly for $c'$ and $j$.  Hence, since $c(e\setminus{v}) = c'(e\setminus{v})$, for any hyperedge
$\mathbbm{1}(i \not \in c(e\setminus \{v\})) -  \mathbbm{1}(j \not \in c(e\setminus \{v\})) =
|c(e)|-|c'(e)|$ and the conclusion follows.
\end{proof}

Lemma~\ref{potential-lem} combined with Claim~\ref{claim_phi} ensures the existence of a pure Nash equilibrium coloring.

Turning to the question of the price of anarchy for CF-vertices, we now prove Theorem~\ref{poa-cf-new} which provides bounds for the family $\mathbb{H}_{r}$, of $r$-uniform hypergraphs. The proof of Theorem~\ref{poa-cf-new} will make use of the following notations:
Fix a hypergraph $H=(V,\E) \in \mathbb{H}_{r}$ and a corresponding coloring, $c$.
 For every triplet $(e,v,i) \in \E \times V \times [k]$ we define the following indicators:
\begin{itemize}
\item
$L(e,v,i) =1$ if and only if $v\in e$, $|c(e\setminus\{v\})| =|c(e)|-1$ and $i\in c(e\setminus\{v\})$.
Otherwise $L(e,v,i) =0$.
In words, $L(e,v,i)=1$ indicates that $v$ gets a utility of one from $e$ and will lose it upon deviation to the color $i$.
\item
$G(e,v,i) =1$ if and only if $v\in e$, $c(e\setminus\{v\}) =c(e)$ and $i\not\in c(e\setminus\{v\})$.
Otherwise $G(e,v,i) =0$.
In words, $G(e,v,i)=1$ indicates that $v$ gets no utility from $e$ but will gain one upon deviation to the color $i$.
\item
$M(e,v,i) =1$ if and only if $v\in e$, $|c(e\setminus\{v\})| =|c(e)|-1$ and $i\not\in c(e)$.
Otherwise $M(e,v,i) =0$.
In words, $M(e,v,i)=1$ indicates that $v$ gets a utility of one from $e$ and will maintain it upon deviation to the color $i$.
\end{itemize}

In addition, let $j(e)= |\{ v\in e: |c(e\setminus\{v\})| =c(e)  \}|$ denote the number of vertices in $e$ whose color is not unique in $e$. Hence $r-j(e)$ counts the number of vertices that are unique. Note that $|c(e)| \le r-j(e)+\lfloor\frac{j(e)}{2}\rfloor \leq r-j(e) + \frac{j(e)}{2}= r-\frac{j(e)}{2}$. Put $\hat j =  \sum_{e\in \E}j(e)$ and conclude that
\begin{equation}\label{bound on the sum of C(e)}
\sum_{e\in \E} |c(e)| \leq \sum_{e\in \E} r-\frac{j(e)}{2} = r|\E|-\frac{\hat j}{2}.
\end{equation}

We also need the following lemma which provides a bound on ${\hat j}$ when $c$ is a Nash equilibrium:

\begin{lem}
\label{j-hat-bound}
Consider an arbitrary $H=(V,\E) \in \mathbb{H}_{r}$ and an arbitrary Nash equilibrium coloring $c$.  Then
\begin{enumerate}
\item
$\hat j \leq \frac{\cardin{\E}r(r-1)}{\frac{r}{2}+k-1}.$
\item
$SW(c) \geq \cardin{\E}r\frac{2k-r}{2k+r-2}.$
\end{enumerate}
\end{lem}

\begin{proof}
(1) As $c$ is a Nash equilibrium, no vertex $v$ can profit by deviating to some color $i \not=c(v)$. Therefore, for any $v\in V$ and any color $i \not=c(v)$,
$\sum_{e\in \E} L(e,v,i) \ge \sum_{e\in \E} G(e,v,i)$. Summing over the colors and the vertices and changing the order of the summation yields:
$$\sum_{e\in \E}\sum_{\{v:v\in e\} } \sum_i L(e,v,i) \geq \sum_{e\in \E}\sum_{\{v:v\in e\} } \sum_i G(e,v,i).$$

Adding $M(e,v,i)$ on both sides we have:
$$\sum_{e\in \E}\sum_{\{v:v\in e\} } \sum_i (L(e,v,i)+M(e,v,i)) \geq
\sum_{e\in \E}\sum_{\{v:v\in e\} } \sum_i (G(e,v,i)+M(e,v,i)).$$
Note that the left-hand side of the inequality satisfies:
\begin{equation}\label{left-side}
\sum_{e\in \E}\sum_{\{v:v\in e\} } \sum_i (L(e,v,i)+M(e,v,i)) = \sum_{e\in \E} (r-j(e))(k-1) = |\E|r(k-1) -{\hat j}(k-1)
\end{equation}
while the right-hand side satisfies the following:
\begin{equation}\label{right-side}
 \sum_{e\in \E}\sum_{\{v:v\in e\} } \sum_i (G(e,v,i)+M(e,v,i)) = \sum_{e\in \E}r(k-|c(e)|)= rk\cardin{\E}-r\sum_{e \in \E}|c(e)|.
\end{equation}

By Inequality~\ref{bound on the sum of C(e)} this latter quantity is greater than or equal
$|\E|r(k-r)+{\hat j}\frac{r}{2}$.

Hence, $|\E|r(k-1) -{\hat j}(k-1) \geq |\E|r(k-r)+{\hat j}\frac{r}{2}$. Rearranging terms we obtain the asserted bound.

(2) Note that for any $v\in V$ and for any $i\not =c(v)$ the sum $\sum_{e\in \E} L(e,v,i)+M(e,v,i)$ is the utility of $v$ and for
$i =c(v)$ the sum $\sum_{e\in \E} L(e,v,i)+M(e,v,i)$ equals zero. Therefore,
$\sum_i\sum_{e\in \E} L(e,v,i)+M(e,v,i) = (k-1)u_v(c)$
is $k-1$ times the utility of $v$ from the coloring $c$. Summing over all vertices implies that
$$(k-1)SW(c) = \sum_v\sum_i\sum_{e\in \E} ( L(e,v,i)+M(e,v,i)).$$
Hence, by Equation~\ref{left-side} we have:
$$(k-1)SW(c) = |\E|r(k-1) -{\hat j}(k-1).$$

Dividing by $k-1$ on both sides and resorting to the  upper bound we have obtained for $\hat j$ in part 1 of the lemma we can conclude that $SW(c) \geq |\E|r-\frac{|\E|r(r-1)}{\frac{r}{2}+k-1}$. Rearranging terms yields the asserted inequality. This completes the proof.
\end{proof}

We are now ready to proceed with the proof of the upper bounds of Theorem~\ref{poa-cf-new} which follows easily from Lemma~\ref{j-hat-bound}:
\begin{proof}
Let $H$ be an arbitrary $r$-uniform hypergraph. Let $c$ be a Nash-equilibrium $k$-coloring so that $SW(c)$ attends $NE(H)$. By Lemma~\ref{j-hat-bound} we have $SW(c) \geq \cardin{\E}r\frac{2k-r}{2k+r-2}.$
On the other hand for any $k$-coloring $\bar c$, $SW(\bar c) \le |\E|r$, so $O(H)\le |\E|r$ in the case when $k \geq r$ and $SW(\bar c) \le |\E|(k-1)$, so $O(H)\le |\E|(k-1)$ in the case when $k < r$.
So for the case $k \geq r$,
$$PoA(H) = \frac{O(H)}{NE(H)} \leq \frac{|\E|r}{\cardin{\E}r\frac{2k-r}{2k+r-2}} = \frac{2k+r-2}{2k-r}$$
 and for the case $k < r$
 $$PoA(H) = \frac{O(H)}{NE(H)} \leq \frac{|\E|(k-1)}{\cardin{\E}r\frac{2k-r}{2k+r-2}}= \frac{k-1}{r}\cdot \frac{2k+r-2}{2k-r}.$$
This completes the proof of the upper bounds for parts 1 and 2 of Theorem~\ref{poa-cf-new}.
\end{proof}

%

{\bf Proof of the lower bound in part 1 of Theorem~\ref{poa-cf-new}:}\\
We lower bound the price of anarchy using the following construction.
 Consider an $r$-uniform $r$-partite hypergraph $H=(V,\E)$ as follows. $V= \bigcup_{i=1}^r A_i$, where each $A_i = \{v_{i,1},\ldots, v_{i,k}\}$ is a set of cardinality $k$. A hyperedge is any subset of $r$ elements in $V$ consisting of exactly one vertex from each $A_i$ so there are exactly $k^r$ hyperedges and each vertex belongs to exactly $k^{r-1}$ hyperedges. Consider the coloring $c(v_{i,j})=j$. That is, every set $A_i$ is colored with all the $k$ colors.  Note that the utility of all vertices equals $(k-1)^{r-1}$. It is easily seen that when a vertex, say $v_{i,j}$, changes its color to, say $l$, then its utility does not change. Therefore, this coloring is a Nash equilibrium. Thus, the social welfare of this Nash equilibrium is $(kr)(k-1)^{r-1}$.

 On the other hand consider the coloring where all vertices of the set $A_i$ are colored with $i$ (so we use a total of $r$ of the $k$ given colors). Notice that for this coloring the social welfare of every vertex is equal to the number of
hyperedges containing it, that is $k^{r-1}$. So for this coloring, the social welfare is $(kr)k^{r-1}$

Dividing the two gives us the asserted lower bound as
$$PoA(H) = \frac{O(H)}{NE(H)} \geq \frac{(kr)k^{r-1}}{(kr)(k-1)^{r-1}} = (\frac{k}{k-1})^{r-1}$$
where the inequality follows from the fact that $NE(H) \leq (kr)(k-1)^{r-1}$.

{\bf Proof of the lower bound in part 2 of Theorem~\ref{poa-cf-new}:}\\
We  resort to the same example that we use for demonstrating a lower bound for the case $k\geq r$. As before, the coloring
$c(v_{ij})=j$ is a Nash equilibrium with $SW(c)=(kr)(k-1)^{r-1}$.
Consider also the coloring $c'$ where all vertices of the set $A_i$ are colored with $i$ for $1\leq i \leq k-1$ and with $k$ for $k \leq i \leq r$. Note that each vertex in $A_i$ ($1\leq i \leq k-1$) gets a utility of $k^{r-1}$ and each other vertex gets a utility zero. So $SW(c') = k(k-1)k^{r-1}$. Hence: $$PoA(H) = \frac{O(H)}{NE(H)} \geq \frac{k(k-1)k^{r-1}}{kr(k-1)^{r-1}} = \frac{k-1}{r}\left (\frac{k}{k-1}\right )^{r-1}$$
 as asserted.

{\bf Proof of part 3 of Theorem~\ref{poa-cf-new}:} Let $H$ be the hypergraph with $r$ vertices -- all forming the unique hyperedge of $H$.
The socially optimal coloring assigns $k-1$ vertices a unique color and the rest of the vertices the remaining color. On the other hand in a coloring which assigns vertices $(i,2i)$ the color $i$, for $i=1,\ldots,k$ and vertices $2k+1,\ldots,r$ color $1$ forms a Nash equilibrium coloring where no agent gains and so its social welfare is zero.
$\Box$

\section{Discussion and Open Problems}
\label{sec:discussion}
In this paper we consider some natural extensions of  mis-coordination games on graphs to hypergraphs and study equilibrium existence and the price of anarchy. The traditional literature on hypergraph coloring considers centralized algorithms for assigning colors and ignores incentives associated with the vertices. In this paper we study how the overall efficiency is impacted when each vertex is associated with a self-interested agent and the choice of colors is delegated to the individual vertices.

For non-monochromatic seeking agents we provide a tight bound on the PoA for $r$-minimal hypergraphs. In particular, the bound we obtain demonstrates that there is almost no loss of social welfare when decisions are decentralized as long as either the number of colors or the size of the minimal hyperedge is large enough. For conflict-free seeking agents the bound on the PoA we provide is not tight and hence calls for further research.  When the size of the hyperedges is roughly $\alpha$ times the number of available colors ($r=\alpha k$) ) and is large we provide an upper bound of $\frac{2+\alpha}{2-\alpha}$ whenever $\alpha\le 1$ and a bound of $\frac{2+\alpha}{\alpha(2-\alpha)}$ for  $1<\alpha<2$ (for $\alpha\ge 2$ the PoA is infinite). On the other hand the lower bounds we have for large $k$ and $r$ are $\exp^\alpha$ for the former case and $\frac{\exp^\alpha}{\alpha}$ for the latter case. In particular for the case $\alpha=1$ (namely, $k=r$) we have $e \le PoA \le 3$.  We hope that further research will help close this gap.

Recall that an instance of a hypergraphical clustering game is specified by a constituent symmetric game for $r$ players (where $r$ is the size of the hyperedge) and a hypergraph. The PoA studies the ratio between the social welfare in an optimal assignment and that of a worst case Nash equilibrium assignment. A natural choice for a social welfare function is the sum of the utilities; however other natural social welfare functions are plausible. The focus of this paper is on two specific utility functions (non-monochromatic and conflict-free seeking). However, the model of a hypergraphical clustering game lends itself to a variety of interesting research questions, some of which we discuss below.  We view this paper as a humble stepping stone to a potentially rich research domain.

Some natural utility functions left for future research are:
\begin{itemize}
\item
Consider the case where a vertex $v$ enjoys a utility of $1$ from each hyperedge in which it is the {\bf unique} vertex with a unique color. Each such hyperedge represents a situation where the vertex has some monopoly power it can exert.
Unfortunately, nothing meaningful can be said for any pair $r,k$ such that $r>2 \ \wedge \ k>2$ or $r>3\ \wedge \ k=2$.
Indeed, consider a hypergraph with only one hyperedge $e$ with $r$ vertices. Any coloring which has at least two unique colors in $e$ (say all vertices are colored by $3$ except for two vertices $x$ and $y$ such that $x$ is colored with $1$ and $y$ is colored with $2$) forms a Nash equilibrium for which no vertex gains more than zero. On the other hand coloring one vertex by $1$  and all the rest by color $2$ provides a social welfare equal to $1$. Since the ratio between the two is infinite the price of anarchy is unbounded. Nevertheless the {\em price of stability} may yield interesting results.%
\footnote{The price of stability, introduced in \cite{Anshelevich:2003:NND:780542.780617,Correa:2004:SRC:1045756.1045783}, considers the ratio between the optimal social welfare and that obtained in the best equilibrium outcome.}
\item
Assume a vertex enjoys a positive utility if its color is unique; however the utility is proportional to the number of such vertices in an edge. This represents a setting where only vertices of a unique color can enjoy some benefit but this benefit is distributed equally among all those unique ones.
\item
A third natural candidate for a utility function is one where each vertex enjoys a hyperedge whenever that hyperedge has some vertex with a unique color (or maybe even a single vertex with a unique color).
\end{itemize}

Most of the literature on the price of anarchy,  similar to our approach, identifies the societal objectives (the social welfare) with the sum of agents' utilities. However, alternative formulations (resulting in a different PoA) are often more adequate.  Consider, the motivating example of frequency assignment. The overall objective in that case could be to provide service to as many customers as possible. This would translate to maximizing the number of hyperedges containing a vertex whose color is unique. However, from an individual vertex' point of view it would like to maximize the number of hyperedges for which its own color is unique (CF-seeking). Note that in this case the social welfare does not coincide with the societal objective function and consequently the bounds on the induced PoA could be different from those obtained in Theorem \ref{poa-cf-new}.

The choice of domain of hypergraphs is another modeling choice that is orthogonal to the specification of utilities and a social welfare function. In this paper we paid attention to $r$-uniform and $r$-minimal hypergraphs; however there are other natural families of hypergraphs that are of interest. For example, geometric hypergraphs induced by, say, discs in the plane (see, e.g., \cite{CF-survey}) which arise in the context of frequency assignments in wireless networks.

In all of the above the choice of colors is given at the outset. However, there may be applications where the number of colors available to the players is regulated and the social objective is to minimize the number of colors accessible to players. In that case one could study the ratio between the number of colors required for a desired assignment in a centralized vs. the decentralized approach. One can then imagine a natural concept of the {\em `coloring burden of anarchy'}, dual to the PoA. For example, the `coloring burden of anarchy' in the context of a conflict-free coloring is the ratio between the number of colors required to obtain a conflict-free coloring in an equilibrium and that number whenever this is dictated centrally ($\cf(H)$). More broadly the notion of  a `coloring burden of anarchy' could refer to the ratio in the number of colors for obtaining some social criterion between the decentralized and the centralized cases.

\subparagraph*{Acknowledgements}
We wish to thank Chaya Keller for helpful comments on this manuscript. We also wish to thank Ron Holzman for pointing out the lower bound construction in Theorem~\ref{poa-cf-new}.

\bibliographystyle{plain}
\bibliography{biblio}

\end{document}